\documentclass[conference]{IEEEtran}
\IEEEoverridecommandlockouts

\usepackage{amsmath,amssymb,amsfonts}
\usepackage{amsthm}
\usepackage{balance}

\newtheorem{theorem}{Theorem}[section]

\newtheorem{proposition}{Proposition}[section]

\usepackage{algorithm}
\usepackage{algpseudocode} 
\usepackage{booktabs}
\usepackage{graphicx}
\algrenewcommand{\algorithmicrequire}{\textbf{Require:}}
\algrenewcommand{\algorithmicensure}{\textbf{Ensure:}}

\algrenewcommand\algorithmicrequire{\textbf{Require:}}
\algrenewcommand\algorithmicensure{\textbf{Ensure:}}

\usepackage{cite}
\usepackage{graphicx}
\usepackage{textcomp}
\usepackage{xcolor}
\usepackage{url}

\begin{document}

\title{Physics-grounded Mechanism Design for \\ Spectrum Sharing between Passive and Active Users\\

\thanks{This work was supported in part by the NSF, under grants
AST-2229103 and AST-2229104.  }
}

\author{
\IEEEauthorblockN{Jiguang Yu}
\IEEEauthorblockA{\textit{SE Division} \\
\textit{Boston University}\\
Boston, Massachusetts \\
jyu678@bu.edu}
\and
\IEEEauthorblockN{Nicholas Brendle}
\IEEEauthorblockA{\textit{ECE Department} \\
\textit{The Ohio State University}\\
Columbus, Ohio \\
brendle.21@osu.edu}
\and
\IEEEauthorblockN{Joel T. Johnson}
\IEEEauthorblockA{\textit{ECE Department} \\
\textit{The Ohio State University}\\
Columbus, Ohio \\
johnson.1374@osu.edu}
\and
\IEEEauthorblockN{David Starobinski}
\IEEEauthorblockA{\textit{ECE Department} \\
\textit{Boston University}\\
Boston, Massachusetts \\
staro@bu.edu}
}

\maketitle

\begin{abstract}
 We propose a physics-grounded mechanism design for dynamic spectrum sharing that bridges the gap between radiometric retrieval constraints and economic incentives. We formulate the active and passive users coexistence problem as a Vickrey-Clarke-Groves (VCG) auctions mechanism, where the radiometer dynamically procures "quiet" time-frequency tiles from active users based on the marginal reduction in retrieval error variance. This approach ensures allocative efficiency and dominant-strategy incentive compatibility (DSIC). To overcome the computational intractability of exact VCG on large grids, we derive an approximation algorithm by using the monotone submodularity induced by the radiometer equation. AMSR-2–based simulations show that the approach avoids high-cost tiles by aggregating low-cost spectrum across time and frequency. In an interference-trap case study, the proposed framework reduces procurement costs by about 60\% over a fixed-band baseline while satisfying accuracy targets.
\end{abstract}


\section{Introduction}
\label{sec:intro}

Passive microwave radiometers aboard the Earth Exploration-Satellite Service (EESS) provide key observations of geophysical variables such as integrated water vapor, near-surface wind speed, and sea-surface temperature~\cite{ITU2017Handbook,nasa_aqua_mission_amsr-e_2001}. These instruments operate in frequency bands that increasingly overlap with active terrestrial services. With radio spectrum becoming increasingly congested, the traditional model of passive sensing (large frequency bands are set aside exclusively and permanently) is untenable, creating friction between scientific observation and commercial use~\cite{national_research_council_spectrum_2010}.

\begin{figure}[t]
    \centering
    \includegraphics[width=1\linewidth]{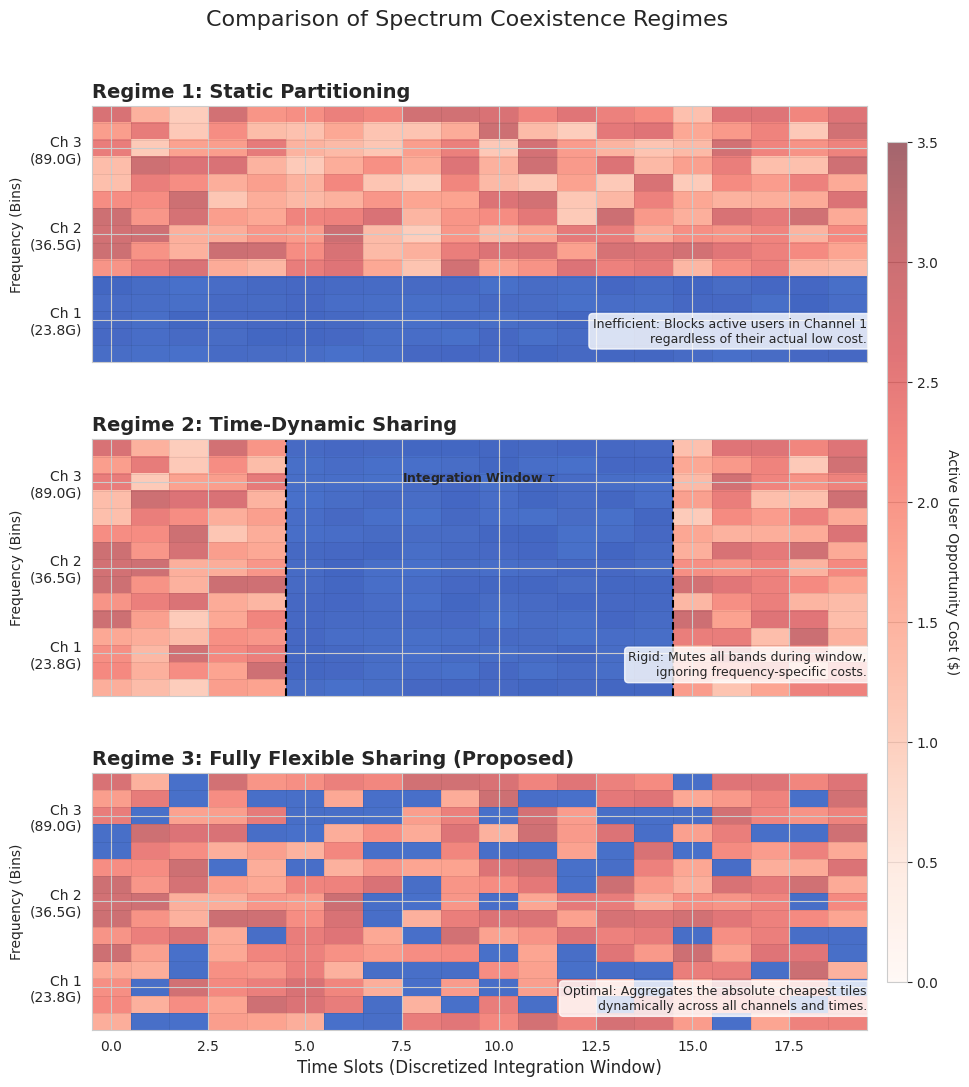}
     \caption{Comparison of coexistence paradigms on a discretized time–frequency grid with three channels. 
    \textbf{(Top) Regime 1: Static Partitioning.} The 23.8 GHz band (Ch 1) is permanently allocated to passive sensing (blue), resulting in the complete exclusion of active users in this band regardless of their costs.
    \textbf{(Middle) Regime 2: Time-Dynamic Sharing.} Spectrum access is regulated by a fixed integration window $\tau$. All channels are silenced simultaneously during the overpass, blocking high-value active users in Ch 2 and Ch 3 while failing to use low-cost periods outside the window.
    \textbf{(Bottom) Regime 3: Fully Flexible Sharing (Proposed).} The market-based mechanism dynamically selects a subset of tiles with the lowest opportunity costs (lightest red background) scattered across the grid.}
    \label{fig:sharing}
\end{figure}

This trade-off lead to a progression toward flexible coexistence. Fig.~\ref{fig:sharing} summarizes these sharing regimes over a radiometric integration window. While static partitioning globally excludes active systems and yields poor spectral efficiency, fully flexible time--frequency sharing allows the radiometer to dynamically procure measurement passbands based on the instantaneous opportunity costs of active users~\cite{Brendle2024IGARSSSharing}. 

We formalize this flexibility using the basic physics of measurement. During a single integration window $\tau$, the time-frequency plane is discretized into tiles of bandwidth $\delta f_x$ and duration $\delta t_x \le \tau$. Each tile $x$ has a clean duty cycle $\alpha_x \in [0,1]$, capturing partial-time muting or hardware latencies. For a procured set of quiet tiles $S$ temporally coincident with $\tau$~\cite{JohnsonEtAl2021JSTARS5001400}, the effective clean bandwidth $B_j(S)$ for channel $j$ is the baseline protected bandwidth plus the normalized allocated spectral volume~\cite{regulations2016radio}. The radiometer equation dictates that thermal-noise variance $\sigma_j^2$ is inversely proportional to the time-bandwidth product $B_j(S)\tau$~\cite{Kaisti2016,JohnsonEtAl2021JSTARS5001400}. This allows quiet spectrum to be flexibly aggregated across time and frequency. Nontrivial cross-channel elasticity ensures comparable product accuracy can be maintained under varying allocations~\cite{BrendleEtAl2024IGARSSAlternateSharing}.

Multi-channel retrievals map these spectral resources directly to mission performance. Because geophysical products are typically estimated as weighted linear combinations of channel brightness temperatures~\cite{Rodgers2000}, the total retrieval error variance is the weighted sum of uncorrelated channel noise variances~\cite{Kay1993}. This linkage quantifies the marginal value of spectrum as the retrieval-error reduction per additional Hz of effective bandwidth. Treating quiet spectrum as a dynamically procured commodity rather than a permanent entitlement raises three fundamental questions:

\begin{enumerate}
    \item How should the radiometer quantify its valuation for quiet tiles given multi-channel coupling and accuracy constraints?
    \item How can a mechanism elicit private cost information from active users to maximize social welfare while ensuring incentive compatibility?
    \item Can scalable algorithms approximate optimal allocations on large time--frequency grids?
\end{enumerate}

This work addresses these questions through the following contributions:

\begin{enumerate}
    \item \textbf{Physics-Driven Valuation Model:} We derive the radiometer's utility explicitly from retrieval physics~\cite{Rodgers2000,Maahn2020}, establishing a closed-form expression for the marginal value density of clean bandwidth based on mask constraints and retrieval errors.
    \item \textbf{VCG-Based Procurement Formulation:} We frame the problem as a multi-unit reverse auction with a tailored VCG mechanism, guaranteeing dominant-strategy incentive compatibility (DSIC) and allocative efficiency~\cite{Vickrey1961,Clarke1971,Groves1973}.
    \item \textbf{Scalable Approximation via Submodularity:} To bypass the computationally complex VCG benchmark, we leverage the radiometer equation~\cite{Kaisti2016} to show that the variance-reduction objective is monotone submodular ($\sigma^2 \propto 1/B$)~\cite{NemhauserWolseyFisher1978}. We design a scalable greedy posted-price algorithm with an analytical approximation bound. Numerically, this achieves a 4.45\% average optimality gap and reduces procurement costs by approximately 60\% in an ``interference trap'' scenario compared to static allocation.
\end{enumerate}

We envision that regulators would procure EESS users with credits or a virtual currency to participate in the bids. Active users winning the auction could be rewarded with subsidies or tax rebates from the government~\cite{merwaday2018incentivizing}. Design of such a reward system requires care, and should be considered as part of future work.

\begin{table}[t]
\centering
\caption{Definition of Parameters and Variables}
\label{tab:nomenclature}
\renewcommand{\arraystretch}{1.2} 
\resizebox{\linewidth}{!}{
\begin{tabular}{l l l}
\toprule
\textbf{Symbol} & \textbf{Definition} & \textbf{Unit / Domain} \\
\midrule
\multicolumn{3}{l}{\textit{Sets and Indices}} \\
$\mathcal{N}, i$ & Set of active users (sellers) and user index & $\{1, \dots, N\}$ \\
$\mathcal{J}, j$ & Set of radiometer channels and channel index & $\{1, \dots, J\}$ \\
$\Omega, x$ & Set of time--frequency tiles and tile index & Set of available time-frequency tiles \\
$S$ & Allocation set (subset of tiles designated as ``quiet'') & $S \subseteq \Omega$ \\
$\mathcal{A}$ & Set of feasible allocations satisfying accuracy constraints & $\mathcal{A} \subseteq 2^\Omega$ \\
\midrule
\multicolumn{3}{l}{\textit{Physics and Retrieval Model}} \\
$\tau$ & Radiometric integration window duration & s \\
$\delta f_x, \delta t_x$ & Frequency width and time duration of tile $x$ & Hz, s \\
$\alpha_x$ & Duty cycle (fraction of quiet time) for tile $x$ & $[0, 1]$ \\
$B_j(S)$ & Effective clean bandwidth for channel $j$ & Hz \\
$\kappa_j$ & Radiometer system noise constant ($\propto T_{\mathrm{sys}}^2$) & K$^2 \cdot$ Hz $\cdot$ s \\
$\sigma_j^2$ & Total measurement error variance in channel $j$ & K$^2$ \\
$c_{k,j}$ & Retrieval sensitivity coefficient for product $k$, channel $j$ & Unit$_k$ / K \\
$\mathrm{Var}[\hat{y}_k]$ & Retrieval error variance for geophysical product $k$ & Unit$_k^2$ \\
$\varepsilon_k^2$ & Maximum permissible error variance (Mission Constraint) & Unit$_k^2$ \\
$\phi_j(M_j)$ & Residual RFI penalty function & K$^2$ \\
\midrule
\multicolumn{3}{l}{\textit{Mechanism Design and Optimization}} \\
$v_0(S)$ & Radiometer (Buyer) valuation for allocation $S$ & Currency (\$) \\
$C_i(S_i)$ & Opportunity cost function for seller $i$ & Currency (\$) \\
$W(S)$ & Social welfare function & Currency (\$) \\
$p_i$ & VCG payment to seller $i$ & Currency (\$) \\
$\lambda_k$ & Dual variable (shadow price) for constraint $k$ & \$ / Unit$_k^2$ \\
$v_j$ & Marginal value density of bandwidth in channel $j$ & \$ / Hz \\
$F(S)$ & Discrete submodular utility function & Unitless / Scaled \\
\bottomrule
\end{tabular}
}
\end{table}

\section{Related Work}

This work bridges passive microwave remote sensing, dynamic spectrum access (DSA), mechanism design, and submodular optimization. We review related work to highlight the specific research gaps addressed by our framework.

\subsection{Passive Microwave Radiometry and Retrieval Modeling}

Passive microwave radiometry provides global atmospheric and surface observations~\cite{UlabyLong2014}, employing variational or regression-based inversion to map multi-channel brightness temperatures to geophysical variables~\cite{Blackwell2005}. While observation-error covariances are often simplified as constant diagonal matrices, modern models account for surface and scan dependencies~\cite{English2013}.

\textit{Distinction:} Prior work assumes a fixed observing-system design with prescribed channels and errors. In contrast, we treat spectral resources as dynamic, controllable inputs, quantifying and procuring clean bandwidth based on instantaneous mission needs.

\subsection{RFI Mitigation and Spectrum Regulation}

Passive sensors are highly vulnerable to radio-frequency interference (RFI) from adjacent active services~\cite{MisraDeMatthaeis2014}. Regulators traditionally manage this via static allocations and strict interference thresholds. While recent works propose dynamic protection zones and time-domain blanking~\cite{Niamsuwan2005,GunerJN07,MohammedAksoyPiepmeierJohnsonBringer2016}, they typically model RFI protection as a binary (pass/fail) constraint based on worst-case scenarios.

\textit{Distinction:} Existing binary schemes ignore cross-channel elasticity. Our mechanism internalizes retrieval physics to move beyond binary protection, enabling the system to trade off quieting costs across channels based on their instantaneous contribution to variance reduction.

\subsection{Dynamic Spectrum Access (DSA)}
Dynamic spectrum access (DSA) seeks to enhance the efficiency of spectrum sharing by relaxing static, long-term exclusivity. A classic model is hierarchical access: primary users retain priority, while secondary users transmit opportunistically under incumbent protection rules \cite{ZhaoSadler2007}. For example, \cite{Goldsmith2009} maximizes secondary performance subject to incumbent protection via interference constraints. The passive EESS users are treated as potential ``victims'' that need to be protected with some constraints.

\textit{Distinction:} In contrast, we model the EESS radiometer as a buyer that can pay to procure the quiet spectrum it needs. Demand is dynamic and driven by physics: the objective is to minimize retrieval error subject to mission constraints. This provides a direct linkage between spectrum-sharing decisions and sensing utility by converting protection requirements into an optimization objective.

\subsection{Mechanism Design for Spectrum Allocation}

Auction theory studies resource allocation under strategic behavior and private information \cite{Myerson1981}. Auctions and reverse auctions have been applied to improve spectrum-sharing efficiency. For example, \cite{chamberlain2024facilitating} proposes an economic model for CBRS-like multi-tier access with short-term, pay-as-you-go priority enhancement, integrating queuing and game-theoretic models.

In addition, a study of secondary spectrum markets \cite{GomezWK19} shows that commoditization of spectrum is, on its own, a significant barrier to liquid markets. Consequently, virtualization into finer-granularity time–frequency units (e.g., PRB-like resource blocks) can improve liquidity.

\textit{Distinction:} In contrast, we construct the buyer’s valuation directly from radiometric retrieval physics. By mapping spectral resources to variance reduction, our procurement formulation differs from prior spectrum-sharing mechanisms.

\subsection{Submodular Optimization in Networks}

Submodularity (equivalently, the diminishing returns property) is central to optimizing set functions under combinatorial constraints. It often arises in wireless networking (e.g., sensor selection, link scheduling, influence maximization) because information gain and coverage objectives exhibit diminishing returns.

\textit{Distinction:} We use the fact that thermal-noise reduction induces a monotone submodular variance-reduction objective. This bridges the gap between the computationally expensive VCG benchmark and a scalable implementation. It provides a physics-based justification for greedy allocation in our spectrum-sharing mechanism.

\section{System and Retrieval Model}
\label{sec:system}

This section formalizes the sensing geometry, resource discretization, and the physics-based mapping from spectral resources to geophysical retrieval accuracy.

\subsection{Time--Frequency Discretization and Effective Bandwidth}
Consider a radiometric integration window of duration $\tau$ during a satellite overpass. The radiometer operates over a set of channels $\mathcal{J} = \{1, \dots, J\}$, where each index $j$ denotes a frequency–polarization channel.

We discretize the relevant time--frequency domain into a finite set of non-overlapping tiles $\Omega$. Each tile $x \in \Omega$ has duration  $\delta t_x \le \tau$ and bandwidth $\delta f_x$. Let $j(x) \in \mathcal{J}$ denote the specific channel to which tile $x$ belongs. Active users (sellers) control these tiles. A tile $x$ is designated as ``quiet'' if the active user suppresses emissions to satisfy a prescribed interference mask $M_{j(x)}$ for a duty cycle $\alpha_x \in [0, 1]$.

The fundamental physical resource provided by a quiet tile is its spectral volume $\alpha_x \delta t_x \delta f_x$ [Hz$\cdot$s]. Radiometric sensitivity is determined by the accumulated spectral volume within the integration window. For a selected subset of quiet tiles $S \subseteq \Omega$, the effective clean bandwidth available to channel $j$ is the time-averaged spectral quantity:
\begin{equation}
B_j(S) \triangleq B_j^{(0)} + \frac{1}{\tau} \sum_{\{x \in S \mid j(x)=j\}} \alpha_x \delta t_x \delta f_x \quad [\mathrm{Hz}],
\label{eq:cleanBW}
\end{equation}
where $B_j^{(0)} > 0$ represents the baseline protected bandwidth ensuring system non-singularity.

\subsection{Noise and Residual Interference Model}
Measurement quality in channel $j$  is governed by the total error variance  $\sigma_j^2$, comprising thermal noise and residual RFI.

\emph{Thermal Noise:} By the radiometer equation, thermal-noise variance scales inversely with the time–bandwidth product:
\begin{equation}
\sigma_{j,\mathrm{th}}^2(S) = \frac{\kappa_j}{B_j(S) \cdot \tau},
\label{eq:radiometer}
\end{equation}
where $\kappa_j \propto T_{\mathrm{sys}}^2$ is a system constant derived from the receiver noise temperature and calibration parameters.

\emph{Residual RFI:} Even under mask compliance, residual emissions may leak into the radiometer passband. We model this residual error as a penalty term $\phi_j(M_j)$ explicitly in units of brightness temperature variance $[\mathrm{K}^2]$. Let $S_j^{\max}(f; M_j)$ denote the mask-limited power spectral density. The total residual interference power at the detector is $P_{\mathrm{RFI},j} = \int G_j(f) S_j^{\max}(f; M_j) df$. The function $\phi_j(\cdot)$ maps this power to an equivalent radiometric error:
\begin{equation}
\phi_j(M_j) = \gamma_j P_{\mathrm{RFI},j}(M_j) + \beta_j P_{\mathrm{RFI},j}^2(M_j) \quad [\mathrm{K}^2].
\label{eq:rfi-penalty}
\end{equation}
Here, $\gamma_j$ $[\mathrm{K}^2/\mathrm{W}]$ and $\beta_j$ $[\mathrm{K}^2/\mathrm{W}^2]$ are calibration coefficients. The linear term accounts for the stochastic noise variance increase, while the quadratic term captures the squared bias contribution to the mean-squared error (MSE). We assume $\phi_j(\cdot)$ is convex and decreases under tighter masks $M_j$.

Combining these terms, the total measurement error variance in channel $j$ is:
\begin{equation}
\sigma_j^2\big(B_j(S), M_j\big) = \frac{\kappa_j}{B_j(S) \tau} + \phi_j(M_j).
\label{eq:sigma-total}
\end{equation}

\subsection{Retrieval Error Propagation}

The system estimates geophysical parameters $\{y_k\}_{k=1}^K$. We use a linearized retrieval model for error budgeting:
\begin{equation}
\hat{y}_k = \mathbf{c}_k^\top \mathbf{T}_b = \sum_{j=1}^{J} c_{k,j} T_{b,j},
\label{eq:linear-retrieval}
\end{equation}
where $c_{k,j}$ is the sensitivity of parameter $k$ to channel $j$.

Assuming uncorrelated channel errors, the retrieval variance for geophysical product $k$ is obtained by propagating the channel variances from \eqref{eq:sigma-total}:
\begin{equation}
\mathrm{Var}[\hat{y}_k] = \sum_{j=1}^{J} c_{k,j}^2 \left( \frac{\kappa_j}{B_j(S) \tau} + \phi_j(M_j) \right).
\label{eq:retrieval-var}
\end{equation}

This mapping defines the feasible region for the resource allocation. To satisfy mission requirements, any valid allocation $S$ must satisfy:
\begin{equation}
\mathrm{Var}[\hat{y}_k] \leq \varepsilon_k^2, \quad \forall k \in \{1, \dots, K\},
\label{eq:accuracy-constraint}
\end{equation}
where $\varepsilon_k$ is the maximum permissible root-mean-square error (RMSE) for product $k$.

\section{Mechanism Design Formulation}
\label{sec:mech}

In this section, we model the dynamic coexistence problem as a multi-unit reverse auction (procurement) game. We define the agents, the allocation space, and the valuation functions derived from the physics model. We then present a VCG mechanism and prove its properties in efficiency and incentive compatibility.

\subsection{Agents and Allocation Space}
The market consists of a single buyer (the EESS radiometer) and a set of $N$ sellers (active users), denoted by $\mathcal{N} = \{1, \dots, N\}$.
The allocatable resources are the time--frequency tiles $\Omega$ defined in Section~\ref{sec:system}. An allocation is a subset of tiles $S \subseteq \Omega$ designated to be ``quiet''.
The set $\Omega$ is partitioned among the sellers based on regulatory licensing or geographic location, such that each tile is controlled by exactly one seller. Let $\Omega_i \subset \Omega$ denote the set of tiles controlled by seller $i$. For any global allocation $S$, let $S_i = S \cap \Omega_i$ denote the local contribution of seller $i$.

\subsection{Valuations and Costs}

\paragraph{Buyer Valuation (Physics-Driven)}
The radiometer's valuation function, $v_0: 2^\Omega \to \mathbb{R} \cup \{-\infty\}$, quantifies the scientific utility of an allocation. Unlike ad-hoc utility functions in standard spectrum auctions, $v_0$ is endogenous to the retrieval physics and strictly enforces mission constraints.
Defining a utility scaling factor $\lambda_0$ (e.g., dollar value per unit of variance reduction), the valuation is defined as:
\begin{equation}
v_0(S) = 
\begin{cases}
\displaystyle \lambda_0 \sum_{k=1}^{K} w_k \left( \varepsilon_k^2 - \mathrm{Var}[\hat{y}_k](S) \right), & \text{if } S \in \mathcal{A} \\
-\infty, & \text{otherwise},
\end{cases}
\label{eq:buyer-val}
\end{equation}
where $\mathcal{A} = \{ S \subseteq \Omega \mid \mathrm{Var}[\hat{y}_k](S) \le \varepsilon_k^2, \forall k \}$ is the set of feasible allocations. The set of weights $w_k$ standardizes units among the geophysical quantities retrieved and can set the relative importance of each if desired.
This formulation ensures the mechanism never selects an allocation that violates mission accuracy requirements, while strictly preferring solutions that offer a larger ``safety margin'' below the error threshold.

\paragraph{Seller Costs (Private Information)}
Each seller $i$ incurs an opportunity cost for quieting tiles $S_i$, denoted by $C_i(S_i; \theta_i)$. The parameter $\theta_i$ represents the seller's private type (e.g., traffic load, QoS priority), which is unknown to the radiometer.
We assume $C_i(\cdot)$ is non-decreasing and normalized such that $C_i(\emptyset) = 0$.
The valuation of seller $i$ for an allocation $S$ is simply the negative of their incurred cost:
\begin{equation}
v_i(S; \theta_i) = -C_i(S \cap \Omega_i; \theta_i).
\label{eq:seller-val}
\end{equation}

\textbf{Remark:} In practice, operators can translate localized throughput loss directly into opportunity costs. For example, if muting a specific time-frequency tile results in a data capacity drop of $\Delta C$ bits, and the network values delayed traffic at a price of $\pi$ per bit, the baseline cost to quiet that tile is simply $c_x = \pi \Delta C$.

\subsection{Social Welfare and Feasibility Assumption}
The social welfare of an allocation $S$ is the sum of the buyer's utility and all sellers' valuations:
\begin{equation}
W(S; \theta) \triangleq v_0(S) + \sum_{i=1}^{N} v_i(S; \theta_i) = v_0(S) - \sum_{i=1}^{N} C_i(S_i; \theta_i).
\label{eq:welfare}
\end{equation}
Note that for any infeasible set $S \notin \mathcal{A}$, the welfare is $-\infty$.

\textbf{Assumption 1 (Non-pivotal feasibility):} 
To ensure pivot payments are well defined, we assume no single seller is essential for feasibility. That is, for any seller $i$, the restricted feasible set $\mathcal{A}_{-i} = \{S \in \mathcal{A} \mid S_i = \emptyset\}$ is non-empty. This implies the radiometer can always satisfy the accuracy constraints $\varepsilon_k$ using resources from the remaining $N-1$ sellers, although potentially at a higher cost or lower safety margin.

\textbf{Remark:} Assumption 1 ensures bounded payments. If a single commercial network dominates, it could set the price arbitrarily high. To prevent this monopolistic hold-up, regulators can enforce a reserve price limit \cite{milgrom2004putting}. Should the winning auction price become unacceptably high, such a mechanism caps the payment and preserves the scientific needs of EESS users.

\subsection{Reverse VCG Mechanism}
We use a Direct Revelation Mechanism where each seller reports a type $\hat{\theta}_i$ (equivalently, declares a cost function $\hat{C}_i$). The mechanism consists of an allocation rule $S(\hat{\theta})$ and a payment rule $p(\hat{\theta})$.

\paragraph{Allocation Rule (Efficient benchmark)}
The benchmark allocation selects the feasible allocation that maximizes the reported social welfare:
\begin{equation}
S^*(\hat{\theta}) \in \arg\max_{S \in \mathcal{A}} \left( v_0(S) - \sum_{i=1}^{N} \hat{C}_i(S_i) \right).
\label{eq:vcg-alloc}
\end{equation}

\paragraph{Payment Rule (Incentive Compatibility)}
We apply the Clarke Pivot Rule tailored for procurement. The payment $p_i$ to seller $i$ is calculated as the externality they impose on the economy. Specifically, it is the difference between the social welfare of others when seller $i$ participates versus when they are absent:
\begin{equation}
p_i(\hat{\theta}) = \underbrace{\hat{C}_i(S^*_i)}_{\text{Compensated Cost}} + \underbrace{\left( W(S^*; \hat{\theta}) - W_{-i}(S^*_{-i}; \hat{\theta}_{-i}) \right)}_{\text{Information Rent}},
\label{eq:vcg-pay}
\end{equation}
where $W_{-i}(S^*_{-i}) = \max_{S' \in \mathcal{A}_{-i}} [v_0(S') - \sum_{j \ne i} \hat{C}_j(S'_j)]$ represents the maximum welfare achievable in a counterfactual economy without seller $i$. 
Due to Assumption 1, $W_{-i}$ is finite, ensuring the payment is well-defined.

\subsection{Theoretical Properties}
\begin{theorem}
Under Assumption 1 and direct revelation, the retrieval-based VCG benchmark satisfies:
\begin{enumerate}
    \item \textbf{Dominant-Strategy Incentive Compatibility (DSIC):} Reporting the true cost function $C_i$ is a utility-maximizing strategy for every seller, regardless of the actions of others.
    \item \textbf{Allocative Efficiency:} The mechanism implements the socially optimal allocation that balances retrieval accuracy against commercial opportunity costs.
    \item \textbf{Individual Rationality (IR):} Sellers are guaranteed non-negative utility (payment covers at least the reported cost).
\end{enumerate}
\label{thm:vcg}
\end{theorem}

\begin{proof}
A detailed proof is provided in Appendix \ref{app:proofs}.
\end{proof}

\section{Deployable Posted-Price Procurement Protocol}
\label{sec:posted-price}

The VCG mechanism in Section~\ref{sec:mech} serves as an idealized economic benchmark for allocative efficiency. To overcome the scalability challenges of combinatorial bidding in fast-fading RF environments, we shift to a deployable protocol: a decentralized posted-price market. Active networks simply post take-it-or-leave-it prices, allowing the radiometer to dynamically procure bandwidth via a scalable greedy approximation.

\subsection{Phase I: Continuous Relaxation and Shadow Prices}
We first relax the discrete allocation problem into a continuous planner's problem. Assume the bandwidth $B_j$ can vary continuously. The objective is to minimize the aggregate cost $C_{\mathrm{tot}}(B)$ subject to the mission accuracy constraints defined in \eqref{eq:accuracy-constraint} and we assume $C_{\mathrm{tot}}(B)$ is convex and non-decreasing in $B$, which is consistent with increasing marginal opportunity costs.

The Lagrangian of this continuous problem is:
\begin{align}
\mathcal{L}(B,\boldsymbol{\lambda}) = C_{\mathrm{tot}}(B) + \sum_{k=1}^{K} \lambda_k \left( \mathrm{Var}[\hat{y}_k](B) - \varepsilon_k^2 \right),
\label{eq:Lagrangian}
\end{align}
where $\lambda_k \ge 0$ are the dual multipliers. Since $\mathrm{Var}[\hat{y}_k]$ is convex with respect to $B$ (as $1/B$ is convex), this is a convex optimization problem. It can be solved efficiently via standard interior-point methods to obtain the optimal dual vector $\boldsymbol{\lambda}^*$.

Substituting the variance expression and taking the partial derivative with respect to $B_j$ yields the first-order optimality condition. We define the marginal value density (shadow price) of clean spectrum in channel $j$ as the magnitude of the weighted marginal variance reduction:
\begin{equation}
v_j(\boldsymbol{\lambda}^*, B_j) \triangleq \left| \sum_{k=1}^{K} \lambda_k^* \frac{\partial \mathrm{Var}[\hat{y}_k]}{\partial B_j} \right| = \sum_{k=1}^{K} \lambda_k^* \frac{c_{k,j}^2 \kappa_j}{\tau B_j^2}.
\label{eq:optimality}
\end{equation}
This shadow price $v_j$ serves as a physics-driven price cap: it captures both the sensitivity of the instrument ($\propto 1/B_j^2$) and the mission-criticality of the channel ($\lambda_k^*$).

\subsection{Phase II: Submodularity of Discrete Utility}
We now return to the discrete domain. To guide the tile selection, we construct a scalar utility function $F(S)$ using the optimal duals $\boldsymbol{\lambda}^*$ derived in Phase I.
Let $F(S)$ denote the weighted variance reduction achieved by an allocation $S$:
\begin{equation}
F(S) \triangleq \sum_{k=1}^{K} \lambda_k^* \left( \mathrm{Var}^{\mathrm{base}}_k - \mathrm{Var}[\hat{y}_k](S) \right).
\label{eq:F-def}
\end{equation}
The procurement goal is to find the minimum-cost set $S$ that satisfies the accuracy constraints. This is equivalent to covering a target utility $\Gamma =\sum \lambda_k^* (\mathrm{Var}^{\mathrm{base}}_k - \varepsilon_k^2)$.

\begin{proposition}[Monotone Submodularity]
\label{prop:submodular}
The retrieval utility function $F: 2^{\Omega} \to \mathbb{R}_{\ge 0}$ defined in \eqref{eq:F-def} is \textbf{monotone non-decreasing} and \textbf{submodular}.
\end{proposition}

\begin{proof}
A detailed proof is provided in Appendix \ref{app:submodularity}.
\end{proof}

\subsection{Phase III: Greedy Posted-Price Algorithm}
In Phase III, we assume incumbents publish take-it-or-leave-it ask prices for tiles (denoted $p(x)$).  Given the posted prices and the utility $F(S)$, then EESS procures the selected tiles at posted prices. Note that this Phase is an algorithmic approximation rather than an incentive-compatible mechanism design.

At each step $t$, the algorithm calculates the discrete marginal gain $\Delta_F(x \mid S_{t-1}) = F(S_{t-1} \cup \{x\}) - F(S_{t-1})$ for every available tile $x$. It selects the tile that maximizes the cost-efficiency ratio:
\begin{equation}
x_t^* = \arg\max_{x \notin S_{t-1}} \frac{\Delta_F(x \mid S_{t-1})}{p(x)},
\label{eq:greedy-ratio}
\end{equation}
where $p(x)$ is an exogenous input (e.g., regulated/observable posted prices). The process repeats until the mission accuracy constraints \eqref{eq:accuracy-constraint} are satisfied. In the worst case, this rule scans all remaining tiles per iteration. In practice, because each tile affects only its associated channel bandwidth, marginal gains can be updated incrementally rather than recomputing $F(S)$ from scratch.

This greedy strategy admits theoretical bounds grounded in submodular optimization theory~\cite{wolsey1982analysis}:

\begin{theorem}
Let $S^*$ be the optimal minimum-cost allocation that satisfies the feasibility constraints. The cost incurred by the greedy algorithm, $C(S_{\mathrm{greedy}})$, satisfies:
\begin{equation}
C(S_{\mathrm{greedy}}) \le \left( 1 + \ln \frac{\max_{x} \Delta_F(x \mid \emptyset)}{\Delta_{\min}} \right) C(S^*),
\end{equation}
where the term in parentheses represents the approximation factor, which is logarithmic in the ratio of the maximum to minimum marginal gains.
\end{theorem}
The algorithm in Phase III proceeds by iteratively selecting the tile with the largest value-to-cost ratio, where each step requires evaluating the marginal gain $\Delta F(x\mid S)$ over the remaining candidate tiles. Importantly, because each tile affects only the effective bandwidth of its associated channel, these marginal gains can be computed via incremental updates rather than recomputing $F(\cdot)$ from scratch. This result gives a logarithmic-factor performance guarantee and provides a practical alternative to the exponential-time VCG benchmark.

The analysis above assumes deterministic interference masks. To incorporate stochastic RFI, one can extend the feasibility formulation using a convex risk measure such as Conditional Value-at-Risk (CVaR). When the underlying loss is convex in the decision variables, CVaR admits a standard convex reformulation and yields a tractable program \cite{RockafellarUryasev2000CVaR}. In this case, the resulting dual variables $\lambda_k^*$ adapt to penalize risk-sensitive channels more heavily, biasing the Phase III selection toward more robust allocations.

\section{Numerical Evaluation}
\label{sec:numerical}

We provide visualizations of the proposed spectrum-sharing mechanism and validate it via numerical simulations. We demonstrate three properties: (1) the capability of the mechanism to prioritize spectral resources based on retrieval physics; (2) the allocative efficiency of flexible sharing compared to rigid baselines; and (3) the scalability of the posted-price approximation for large-scale grids.

We build a simulation environment modeled after the Advanced Microwave Scanning Radiometer 2 (AMSR-2) mission to evaluate the mechanism under realistic physical and economic conditions. The setup couples a multi-channel retrieval model with a spatially correlated commercial opportunity-cost field, as visualized in Fig.~\ref{fig:sim_setup}. To stress-test allocative efficiency, we construct an “Interference Trap” scenario. This consists of a specific spectral region where the opportunity costs mimic high-priority commercial services, creating a cost barrier that a naive fixed-band allocation would fail to avoid. The specific simulation parameters are configured as follows:

\begin{figure}[t]
    \centering
    \includegraphics[width=1\linewidth]{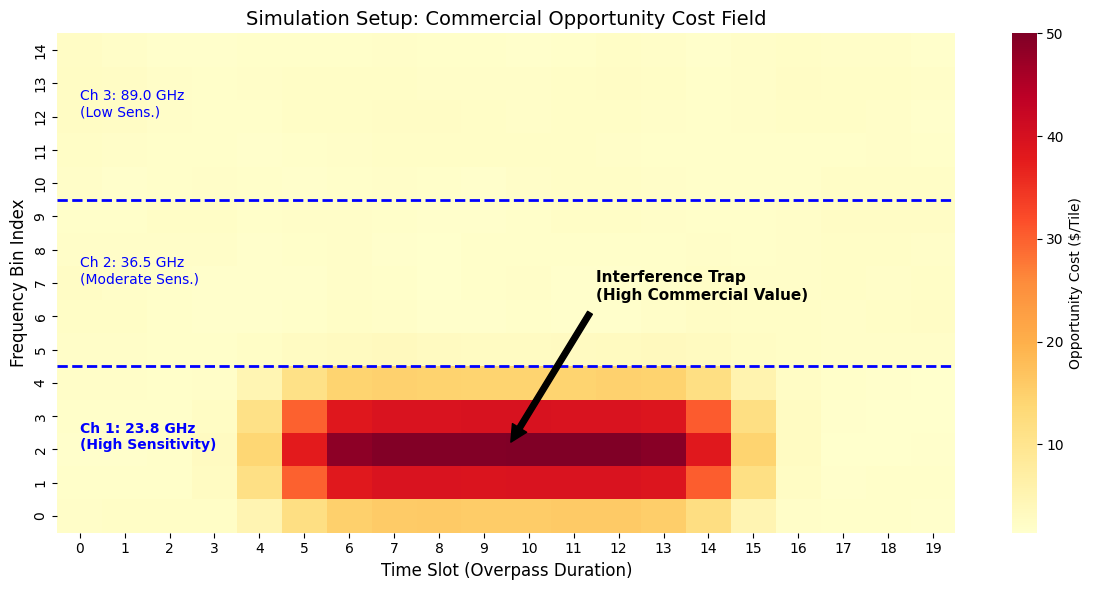}
    \caption{Simulation Setup: Commercial Opportunity Cost Field ($15 \times 20$ grid). The heatmap displays the opportunity cost for each time--frequency tile. The dashed blue lines delineate the three radiometric channels. A high-cost ``Interference Trap'' (dark red, high cost) is considered in the highly sensitive 23.8 GHz band to test the mechanism's ability to navigate spectral congestion.}
    \label{fig:sim_setup}
\end{figure}

    \textbf{1. Radiometer Physics and Simplified Model} We simulate a three-channel system targeting Integrated Water Vapor (IWV), modeled after the AMSR-2 configuration. While the actual mission utilizes dual-polarization feeds (V/H) resulting in six channels, we adopt a simplified effective model focusing on the three primary frequency bands (23.8, 36.5, and 89.0 GHz) to maintain tractability while capturing the core spectral dependencies.
    
    To address the non-linear nature of radiative transfer, the retrieval is formulated as a linearized estimator valid within a specific validity regime. We assume the retrieval operates under a small-perturbation assumption, where the higher-order non-linear terms (bias) in the Taylor expansion are negligible compared to the stochastic noise variance ($\sigma^2 \propto 1/B$). The sensitivity vector is set to $\mathbf{c}_{\mathrm{IWV}} = [0.45, -0.20, 0.05]^\top$, reflecting the primary dependence on the 23.8 GHz water vapor resonance, with system noise constants $\kappa \approx 500$ [K$^2\cdot$Hz$\cdot$s].
    
    \textbf{2. Mission Constraints:} The retrieval accuracy target is set to a maximum error variance of $\varepsilon_{\mathrm{IWV}}^2 = 0.25$ (corresponding to an RMSE of 0.5 units). This constraint defines the feasible region $\mathcal{A}$ for the mechanism.

    \textbf{3. Time--Frequency Grid:} The overpass integration window is discretized into $T=20$ time slots. Each of the three channels comprises 5 frequency bins, resulting in a total $15 \times 20 =300$ tiles.

    \textbf{4. Commercial Environment:} The active user opportunity costs are generated via a Gaussian-smoothed random field. The background cost ranges uniformly between \$1 and \$3 per tile. To simulate a challenging coexistence scenario, we include an ``Interference Trap''. It is a localized high-cost region in the 23.8 GHz band (time slots 5–15) with peak opportunity cost \$50 per tile (see Fig~\ref{fig:sim_setup}).

To illustrate the selection dynamics, we focus on clearing behavior within the 23.8 GHz channel. Figure~\ref{fig:market_clearing} visualizes the interaction between the physics-based demand and the commercial supply.

\begin{figure}[t]
    \centering
    \includegraphics[width=0.9\linewidth]{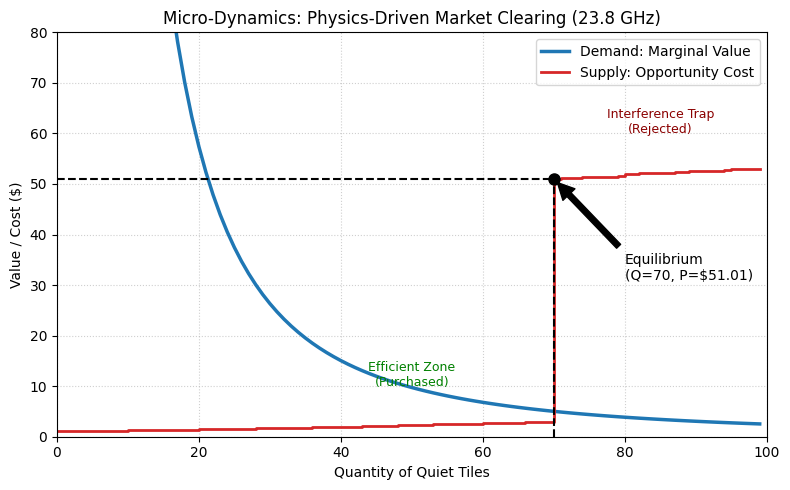}
    \caption{Micro-Dynamics of Market Clearing in the 23.8 GHz Channel. The blue curve illustrates the radiometer's marginal valuation density ($v \propto 1/B^2$), exhibiting diminishing returns. The red steps represent the sorted opportunity costs of active users. The mechanism reaches equilibrium at $Q = 70$ tiles. Note that the equilibrium marker is positioned at the cost of the first rejected tile ($P \approx \$51$).
    This point illustrates the impact of the ``Interference Trap'': the mechanism effectively purchases low-cost background spectrum and autonomously stops procurement precisely when the opportunity cost of the subsequent tile exceeds the marginal scientific value. }
    \label{fig:market_clearing}
\end{figure}

The market clearing process demonstrates three distinct behaviors driven by the valuation function from Section~\ref{sec:mech}. \textbf{High-Value Acquisition behavior.} In the initial phase ($Q < 20$), the clean bandwidth is low ($B_j \approx 0$), causing the marginal value of spectrum to increase very fast ($v_j \to \infty$). The mechanism aggressively procures tiles. \textbf{Cost-Sensitive Equilibrium behavior.} As bandwidth accumulates, the marginal reduction in retrieval variance diminishes. The demand curve intersects the vertical facet of the supply curve. It clears the market by buying up all of the low-cost tiles (up to $P \approx \$3.00$), but the marginal retrieval gain remains less than the next interference trap ( $\$ 50+$ ), resulting in an equilibrium at $Q=70$.

\textbf{Trap Avoidance behavior.} The equilibrium quantity, $Q = 70$, aligns exactly with the boundary of the low-cost background spectrum. The mechanism identifies that the marginal retrieval gain from the subsequent ``Interference Trap'' tiles does not justify their price, thereby maximizing social welfare by leaving the high-traffic spectrum to active users.

To quantify the advantage of our mechanism, we compare the Flexible VCG with a Fixed Band baseline model. In this baseline model, the radiometer is strictly confined to the primary 23.8 GHz channel (Channel 1) and must satisfy mission accuracy requirements regardless of commercial interference costs. We run a simulation of an ``Interference Trap'' scenario where a high-value commercial use increases the availability cost of specific Channel 1 tiles to approximately \$50, while adjacent Channel 2 (36.5 GHz) tiles are still available at \$1. The Fixed Band strategy does not allow substitutions across channels.

\begin{table}[h]
\centering
\caption{Comparison of Social Welfare under "Interference Trap" conditions (Trap Cost $\approx \$50$).}
\label{tab:efficiency_results}
\resizebox{\linewidth}{!}{%
\begin{tabular}{l c c c}
\toprule
\textbf{Metric} & \textbf{Fixed Band} & \textbf{Flexible VCG} & \textbf{Improvement} \\
\midrule
Value Created ($V$) & \$5,113 & \$5,001 & -2.2\% \\
Cost Incurred ($C$) & \$1,290 & \$464 & \textbf{-64.0\%} \\
\textbf{Net Welfare ($W$)} & \textbf{\$3,823} & \textbf{\$4,537} & \textbf{+18.7\%} \\
\bottomrule
\end{tabular}
}
\end{table}

The outcome is summarized in Table~\ref{tab:efficiency_results}. To reach the accuracy target, the Fixed Band strategy is forced to purchase 23 high-cost tiles within the interference trap, incurring a total cost of \$1,290. In contrast, the total cost of Flexible VCG is only \$464. By making full use of cross-channel substitution enabled by the physics-driven valuation, the mechanism identifies that while Channel 2 is physically less sensitive than Channel 1, the marginal cost of noise reduction in Channel 2 (at \$1/tile) is significantly lower than in the trapped Channel 1. Consequently, the mechanism procures the available low-cost Channel 1 tiles outside the trap and then switches to Channel 2 to purchase additional low-priced tiles to compensate for the sensitivity difference. This cross-channel substitution achieves comparable scientific accuracy ($V \approx \$5000$) while reducing procurement costs by 64.0\%, resulting in a net welfare gain of 18.7\%.

Finally, we evaluate the computational performance of the Scalable Posted-Price Algorithm (Section~\ref{sec:posted-price}) against the optimal VCG solver. We formulate the benchmark as a Minimum-Cost Coverage problem: finding the subset of tiles that satisfies the retrieval accuracy target $\mathrm{Var} \le \varepsilon^2$ at the lowest possible cost.
We vary the grid size from 5 to 21 tiles. For each grid size, we generate random instances with uniform bandwidth contributions $b_x \sim U(1, 5)$ and costs $c_x \sim U(1, 10)$.

\begin{figure}[t]
    \centering
    \includegraphics[width=1\linewidth]{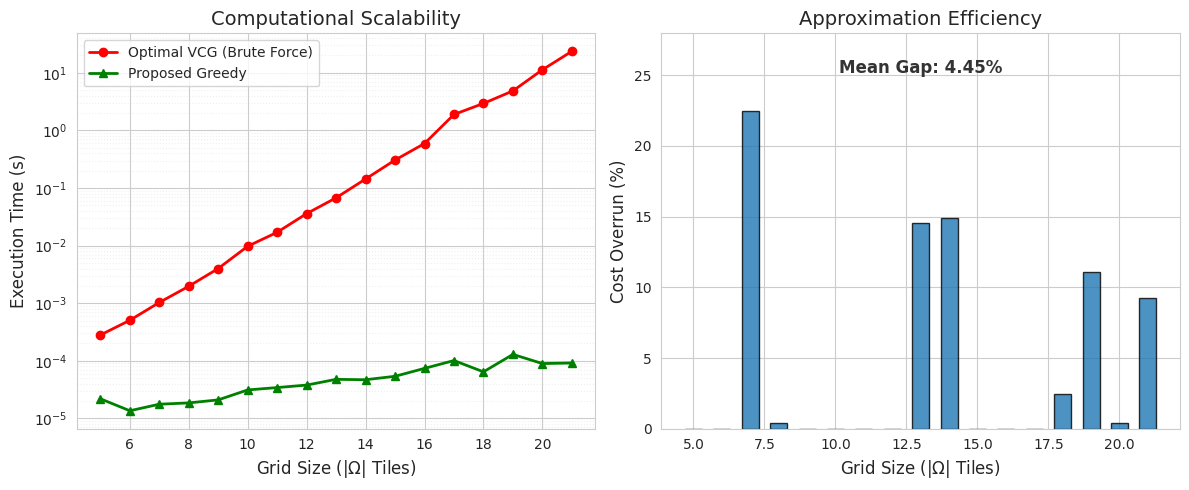}
    \caption{Computational scalability and efficiency. Left: Execution time (log scale) versus the grid size $|\Omega|$. The optimal solver grows exponentially with ($|\Omega|)$) in our implementation, reaching 25 seconds for $|\Omega|=21$. The proposed Scalable Posted-Price Greedy Algorithm (green) scales approximately linearly over the tested grid sizes, solving instances in microseconds. Right: Optimality gap (percentage cost overrun relative to the global optimum). The greedy solution achieves a mean optimality gap of 4.45\%.}
    \label{fig:scalability}
\end{figure}

The results in Fig.~\ref{fig:scalability} highlight the trade-off between optimality and tractability. In our implementation, the optimal solver’s runtime grows exponentially with $|\Omega|$and reaches 25.34 s at $|\Omega|=21$, making exact benchmark computation impractical for high-resolution grids (i.e., $|\Omega| \gg 100$). In contrast, the dual-guided greedy algorithm scales roughly linearly over the tested sizes and returns near-optimal costs. The mean optimality gap is 4.45\%, ranging from 0\% to 22\% in our trials. This behavior is consistent with the logarithmic approximation bound in Theorem V.1 and suggests limited loss in practice.

\section{Conclusion}
\label{sec:conclusion}

This paper introduces a physics-driven market framework for EESS spectrum access, shifting away from static regulatory protection. By explicitly mapping radiometric retrieval error ($\sigma^2 \propto 1/B$) to economic valuation, we model spectrum coexistence as a multi-unit reverse auction. We formulate a VCG mechanism that guarantees dominant-strategy incentive compatibility and allocative efficiency. To overcome the computational complexity of the VCG benchmark on high-resolution grids, we propose a scalable posted-price algorithm that leverages the monotone submodularity of variance reduction to guarantee a logarithmic approximation bound.

Simulations based on AMSR-2 parameters validate this approach. In a heavily congested ``interference trap'' scenario, our flexible mechanism autonomously shift procurement to lower-cost bands. This maintains the scientific accuracy target while reducing procurement costs by approximately 60\% compared to rigid baseline allocations. Ultimately, integrating measurement physics with market economics demonstrates that scientific and commercial spectrum needs are not inherently zero-sum, allowing for a complementary coexistence that maximizes both spectral utility and social welfare.

While our framework establishes a theoretical foundation, it relies on a linearized retrieval model and assumes ideal calibration. Physically, relying on a linearized retrieval model and ideal calibration may underestimate actual error variances in highly non-linear atmospheric regimes. The economic assumptions imply that our simulated cost reductions could degrade in real environments. Future work will address non-linear retrieval dynamics and unmodeled system errors, along with the hardware constraints of fast-switching digital backends. 

Beyond these technical refinements, transitioning our mechanisms into the real world will require an interdisciplinary exploration of legal and regulatory frameworks governing dynamic spectrum markets, alongside empirical validation in physical testbeds.

\appendices

\section{Proof of Mechanism Properties (Theorem~\ref{thm:vcg})}
\label{app:proofs}

Let $W(S; \hat{\theta}) = v_0(S) - \sum_{j} \hat{C}_j(S_j)$ denote the reported social welfare. The mechanism selects $S^* \in \arg\max_{S \in \mathcal{A}} W(S; \hat{\theta})$ and determines payments via the Clarke pivot rule: $p_i = \hat{C}_i(S^*_i) + [W(S^*; \hat{\theta}) - W_{-i}]$, where $W_{-i} = \max_{S \in \mathcal{A}, S_i=\emptyset} (v_0(S) - \sum_{j \ne i} \hat{C}_j(S_j))$.

\paragraph{Dominant-Strategy Incentive Compatibility (DSIC)}
The utility of the seller $i$' is $u_i = p_i - C_i(S^*_i)$. Substituting the payment rule:
\begin{align}
u_i &= \left( \hat{C}_i(S^*_i) + W(S^*; \hat{\theta}) - W_{-i} \right) - C_i(S^*_i) \nonumber \\
    &= \underbrace{\left( v_0(S^*) - \sum_{j \ne i} \hat{C}_j(S^*_j) - C_i(S^*_i) \right)}_{W(S^*; \theta_i, \hat{\theta}_{-i})} - \underbrace{W_{-i}}_{\text{Const}}.
    \label{eq:u_decomp}
\end{align}
The term $W_{-i}$ depends only on $\hat{\theta}_{-i}$. Thus, maximizing $u_i$ is equivalent to maximizing the mixed welfare $W(S^*; \theta_i, \hat{\theta}_{-i})$. Since the mechanism computes $S^*$ to maximize the reported welfare, truthful reporting ($\hat{\theta}_i = \theta_i$) aligns the objective of the mechanism with the objective of the agent, making it a dominant strategy. Allocative Efficiency follows immediately: since agents report truthfully, $S^*$ maximizes the true social welfare. 

\paragraph{Individual Rationality (IR)}
Assume truth-telling ($\hat{C}_i = C_i$). Eq.~\eqref{eq:u_decomp} simplifies to $u_i = \max_{S \in \mathcal{A}} W(S; \theta) - W_{-i}$.
The term $W_{-i}$ is the maximum welfare over the restricted set $\mathcal{A}'_i = \{S \in \mathcal{A} \mid S_i = \emptyset\}$. Since $\mathcal{A}'_i \subseteq \mathcal{A}$, the global maximum over $\mathcal{A}$ must satisfy:
\begin{equation}
\max_{S \in \mathcal{A}} W(S; \theta) \ge \max_{S \in \mathcal{A}'_i} W(S; \theta).
\end{equation}
Thus, $u_i \ge 0$. Assumption 1 ensures $\mathcal{A}'_i \neq \emptyset$, guaranteeing finite payments.

\section{Proof of Monotonicity and Submodularity (Proposition~\ref{prop:submodular})}
\label{app:submodularity}

In this section, we provide a direct proof of the structural properties of the retrieval utility function $F(S)$.

\subsubsection{Structural Decomposition}
Recall the definition of the weighted variance reduction from \eqref{eq:F-def}:
\begin{equation}
F(S) = \sum_{k=1}^{K} \lambda_k \left( \mathrm{Var}^{\mathrm{base}}_k - \sum_{j=1}^{J} c_{k,j}^2 \sigma_j^2(S) \right).
\end{equation}
Since the weights $\lambda_k$ and squared retrieval coefficients $c_{k,j}^2$ are non-negative, and $\mathrm{Var}^{\mathrm{base}}_k$ is a constant, $F(S)$ is a non-negative linear combination of terms dependent on individual channel variances. Thus, it suffices to prove the monotonicity and submodularity for the negative noise variance of a single channel $j$. Let us define the channel utility proxy $u_j(S)$:
\begin{equation}
u_j(S) \triangleq - \frac{1}{B_j(S)}.
\end{equation}
Note that physical proportionality constants (e.g., system noise temperature $\kappa_j$) are strictly positive and do not affect the convexity properties; they are omitted here for clarity. The effective bandwidth is given by:
\begin{equation}
B_j(S) = B_j^{(0)} + \sum_{x \in S : j(x)=j} \Delta b_x,
\end{equation}
where $B_j^{(0)} > 0$ is the baseline protected bandwidth, ensuring the denominator is always strictly positive, and $\Delta b_x > 0$ is the bandwidth contribution of tile $x$.

\subsubsection{Proof of Monotonicity}
Consider a set of tiles $A \subseteq \Omega$ and a new tile $z \notin A$.
\begin{itemize}
    \item If $j(z) \ne j$, the bandwidth is unchanged, so the marginal gain is zero.
    \item If $j(z) = j$, the new bandwidth is $B_j(A \cup \{z\}) = B_j(A) + \Delta b_z$.
\end{itemize}
Since $\Delta b_z > 0$, we have $B_j(A \cup \{z\}) > B_j(A)$. Because the function $g(y) = -1/y$ is strictly increasing for $y > 0$, it follows that:
\[
- \frac{1}{B_j(A) + \Delta b_z} > - \frac{1}{B_j(A)} \implies u_j(A \cup \{z\}) > u_j(A).
\]
Thus, the marginal gain is always non-negative, proving that $F(S)$ is monotone non-decreasing.

\subsubsection{Proof of Submodularity (Diminishing Returns)}
A set function $u_j$ is submodular if for any nested sets $A \subseteq \mathcal{Q} \subseteq \Omega$ and any element $z \notin \mathcal{Q}$, the marginal gain satisfies $\Delta u_j(z \mid A) \ge \Delta u_j(z \mid \mathcal{Q})$.

\textbf{Case 1:} $j(z) \ne j$. The marginal gain is identically zero for both sets, so the inequality holds trivially.

\textbf{Case 2:} $j(z) = j$. Let $b_A = B_j(A)$ and $b_{\mathcal{Q}} = B_j(\mathcal{Q})$. Since $A \subseteq \mathcal{Q}$ and bandwidth contributions are non-negative, we have $0 < b_A \le b_{\mathcal{Q}}$.
The marginal gain of adding tile $z$ to set $A$ is:
\begin{align}
\Delta u_j(z \mid A) &= \left( - \frac{1}{b_A + \Delta b_z} \right) - \left( - \frac{1}{b_A} \right) \nonumber \\
&= \frac{1}{b_A} - \frac{1}{b_A + \Delta b_z} \nonumber \\
&= \frac{\Delta b_z}{b_A (b_A + \Delta b_z)}.
\label{eq:marginal_gain_explicit}
\end{align}
Let $D(b) \triangleq b^2 + b \Delta b_z$ denote the denominator term. Since $b > 0$ and $\Delta b_z > 0$, $D(b)$ is a strictly increasing function of $b$.
Given $b_A \le b_{\mathcal{Q}}$, it implies $D(b_A) \le D(b_{\mathcal{Q}})$. Consequently, the reciprocal relationship yields:
\begin{equation}
\frac{\Delta b_z}{D(b_A)} \ge \frac{\Delta b_z}{D(b_{\mathcal{Q}})} \implies \Delta u_j(z \mid A) \ge \Delta u_j(z \mid \mathcal{Q}).
\end{equation}
This confirms that $u_j(S)$ satisfies the diminishing returns property. Since $F(S)$ is a non-negative linear combination of submodular 

\section*{AI Disclosure Statement}
During the preparation of this work, the authors used Gemini~3 to conduct literature search, refine language, and format the paper. These contents were reviewed and edited by the authors to ensure accuracy.
\balance

\bibliographystyle{IEEEtran}
\bibliography{refs}

\end{document}